\documentclass[11pt]{amsart}

\title{Automorphisms of corona algebras, and group cohomology}

\author{Samuel Coskey}
\address{Samuel Coskey\\Department of Mathematics\\Boise State University\\ 1910 University Dr\\Boise, ID}
\email{scoskey@nylogic.org}
\urladdr{boolesrings.org/scoskey}

\author{Ilijas Farah}
\address{Ilijas Farah\\ Department of Mathematics and Statistics\\
  York University\\ 4700 Keele Street\\ North York, Ontario\\ Canada,
  M3J 1P3; and Matematicki Institut, Kneza Mihaila 34, Belgrade,
  Serbia}
\email{ifarah@yorku.ca} 
\urladdr{http://www.math.yorku.ca/$\sim$ifarah}

\thanks{The second author was partially supported by NSERC}
\usepackage{amssymb}
\usepackage{tikz}
\usepackage{diagrams}

\newtheorem{theorem}{Theorem}[section]
\newtheorem{proposition}[theorem]{Proposition}
\newtheorem{lemma}[theorem]{Lemma}
\newtheorem{corollary}[theorem]{Corollary}
\newtheorem{conjecture}[theorem]{Conjecture}
\theoremstyle{definition}
\newtheorem{definition}[theorem]{Definition}
\newtheorem{hypothesis}[theorem]{Hypothesis}
\newtheorem{question}[theorem]{Question}

\newcommand{\cq}{Q}
\newcommand{\cc}{C}
\newcommand{\NN}{\mathbb{N}}\newcommand{\bbN}{\mathbb{N}}

\newcommand{\TT}{\mathbb{T}}\newcommand{\bbT}{\mathbb{T}}
\newcommand{\cUN}{\mathbb{T}^\mathbb{N}}
\newcommand{\bfA}{\mathbf{A}}
\newcommand{\bfF}{\mathbf{F}}
\newcommand{\bfG}{\mathbf{G}}
\newcommand{\cA}{\mathcal{A}}
\newcommand{\cB}{\mathcal{B}}
\newcommand{\cC}{\mathcal{C}}
\newcommand{\cK}{\mathcal{K}}
\newcommand{\cU}{\mathcal{U}}
\newcommand{\cP}{\mathcal{P}}
\DeclareMathOperator{\Hom}{Hom}
\DeclareMathOperator{\Ext}{Ext}
\DeclareMathOperator{\Ad}{Ad}
\DeclareMathOperator{\Aut}{Aut}
\DeclareMathOperator{\im}{im}
\DeclareMathOperator{\dom}{dom}
\newcommand{\ilim}{\varprojlim\nolimits}
\newcommand{\minus}{\mathord{\smallsetminus}}
\newcommand{\id}{\mathrm{id}}
\newcommand{\rs}{\restriction}
\newcommand{\abs}[1]{\left|#1\right|}
\newcommand{\norm}[1]{\left\|#1\right\|}
\newcommand{\set}[1]{\left\{\,#1\,\right\}}

\newenvironment{eqtext}
  {\begin{equation}\begin{minipage}{.9\linewidth}}
  {\end{minipage}\end{equation}\ignorespacesafterend}

\begin{document}

\begin{abstract}
  In 2007 Phillips and Weaver showed that, assuming the Continuum
  Hypothesis, there exists an outer automorphism of the Calkin
  algebra.  (The Calkin algebra is the algebra of bounded operators on
  a separable complex Hilbert space, modulo the compact operators.)
  In this paper we establish that the analogous conclusion holds for a
  broad family of quotient algebras.  Specifically, we will show that
  assuming the Continuum Hypothesis, if $A$ is a separable algebra
  which is either simple or stable, then the corona of $A$ has
  nontrivial automorphisms.  We also discuss a connection with
  cohomology theory, namely, that our proof can be viewed as a
  computation of the cardinality of a particular derived inverse
  limit.
\end{abstract}

\maketitle

\section{Introduction}

The Calkin algebra came into prominence as the ambient structure for
the BDF theory in the seminal work of Brown--Douglas--Fillmore
\cite{BrDoFi:Unitary,BrDoFi:Extensions} (see \cite{Dav:C*} for an
exposition).  In their latter article \cite{BrDoFi:Extensions}, the
authors asked whether the Calkin algebra has any outer automorphisms.
The question remained open for more than thirty years, when it was
shown to be independent from the axioms of set theory by
Phillips--Weaver \cite{PhW} and the second author \cite{Fa:All}.  In
this paper, we are interested in a generalization of this question to
corona algebras (sometimes called outer multiplier algebras).  Both
the original question and the generalization naturally belong to the
program of analyzing outer automorphism groups of quotient structures,
a program pursued by the second author for over a decade (see
\cite{Fa:All}, \cite{Fa:Rigidity}, and \cite{Fa:CH}, among others).

Formally, if $A\subset B(\mathcal H)$ and the annihilator of $A$ is
trivial, then the \emph{multiplier algebra} of $A$ is the set $M(A)$
consisting of all $m\in B(\mathcal H)$ such that both $mA$ and $Am$
are contained in $A$.  It is well-known that $M(A)$ does not depend on the
representation of $A$. Then $A$ is a two-sided, norm-closed, ideal of
$M(A)$ and the \emph{corona} of $A$ is simply the quotient $M(A)/A$.
In this paper, the corona of $A$ will be denoted by $\cq(A)$ in order to 
avoid confusion with an algebra of continuous functions on a
compact Hausdorff space $X$, denoted by $\cc(X)$.

The corona construction can be viewed as a noncommutative analogue of
the \v Cech--Stone remainder of a topological space, since if
$A=\cc(X)$ then the corona of $A$ is given by $\cc(\beta X\minus
X)$.  (Here, $\beta X\minus X$ is the \v Cech--Stone remainder of
$X$).  Corona algebras also generalize the Calkin algebra, since it is
easily seen that the corona of the algebra $\mathcal K$ of compact
operators is exactly $\cq(\mathcal H)=B(\mathcal H)/\mathcal
K$.  Like the Calkin algebra, coronas have played a role in the
literature; for instance they provide the ambient structure for the
Busby invariant for extensions of C*-algebras \cite[\S
II.8.4.4]{Black:Operator} (see also \cite{Pede:Corona}).

We wish to generalize the result that the Calkin algebra can have
outer automorphisms to the case of more general corona algebras.  In
our results, we will generalize ``outer'' to the more restrictive
notion of ``nontrivial''.  In fact, the definition of ``trivial''
below is arguably the most comprehensive definition that is reasonable
(see Section~\ref{S.trivial} for discussion).  The only property of
trivial automorphisms that we shall need in our main results is that
if $A$ is a separable C*-algebra, then $\cq(A)$ has at most
$2^{\aleph_0}$ trivial automorphisms.

\begin{definition}
  \label{Def.trivial}
  An automorphism $\Phi$ of a separable C*-algebra $A$ is said to be
  \emph{trivial}  if the set 
  \[
  \Gamma_\Phi=\{(a,b)\in M(A)^2: \Phi(a/A)=b/A\}
  \]
  is Borel, where $M(A)$ is endowed with the strict topology. 
\end{definition}


As we shall see in Section~\ref{S.trivial}, trivial automorphisms form
a group, all inner automorphisms are trivial, and in case of the
Calkin algebra all trivial automorphisms are inner.  However in some
coronas there are trivial automorphisms which are not inner.  For
example, if $A=\cc_0(X)$ is an abelian C*-algebra then the corona
$\cq(A)$ has no inner automorphisms.  However, many trivial
automorphisms arise from homeomorphisms between co-compact subsets of
$X$. Conjecturally no other automorphisms of such $\cq(A)$ can be
constructed without use of additional set-theoretic axioms such as the
Continuum Hypothesis (see \cite[\S 4]{Fa:AQ}).

We now arrive at the question of whether corona algebras have
nontrivial automorphisms.  In addition to the case of the Calkin
algebra, the analogous question has been answered in several other
categories (see the survey in~\cite{Fa:Rigidity}).  Based on these
answers, it is natural to consider the following two conjectures. 

\begin{conjecture}
  \label{conj:ch}
  The Continuum Hypothesis implies that the corona of every separable,
  non-unital C*-algebra has nontrivial automorphisms.
\end{conjecture} 

\begin{conjecture}
  \label{conj:rigid}
  Forcing axioms imply that the corona of every separable, non-unital
  C*-algebra has only trivial automorphisms.
\end{conjecture}

Of course, Conjecture~\ref{conj:rigid} can be made stronger with a
more restrictive definition of ``trivial automorphism.''  Still, a
confirmation even of this weak form would be a remarkable achievement.
In this paper we will give a confirmation of Conjecture~\ref{conj:ch}
for a large class of corona algebras.

\begin{theorem}
  \label{thm:main}
  Assume that the Continuum Hypothesis holds.  Let $A$ be a
  $\sigma$-unital C*-algebra of cardinality $2^{\aleph_0}$ such that:
  \begin{enumerate}
  \item $A$ is simple, or
  \item $A$ is stable, or
  \item $A\cong B\otimes C$, where $C$ is non-unital and
    simple, or
  \item $A$ has a non-unital, $\sigma$-unital quotient
    with a faithful irreducible representation.
  \end{enumerate}
  Then the corona of $A$ has $2^{2^{\aleph_0}}$ many automorphisms.
\end{theorem}

In the case when $A$ is separable, the corona of $A$ has at most
$2^{\aleph_0}$ many trivial automorphisms.  By the classical result of
Cantor that $\kappa<2^\kappa$ for every cardinal $\kappa$, and we have
the following consequence.

\begin{corollary}
  Suppose that the Continuum Hypothesis holds, and that $A$ is a
  separable, non-unital, C*-algebra satisfying any of (1)--(4) above.
  Then the corona of $A$ has nontrivial automorphisms.
\end{corollary}

Before discussing the proof of Theorem~\ref{thm:main}, let us put it
in context by reviewing the cases in which Conjectures~\ref{conj:ch}
and~\ref{conj:rigid} have been answered.  The first result is due to W.\
Rudin \cite{Ru}, who confirmed Conjecture~\ref{conj:ch} for the corona
algebra $\cc(\beta\NN\minus\NN)$.  Of course this conclusion is only
in hindsight; in fact, Rudin established the topological
reformulation: there exist nontrivial homeomorphisms of the \v
Cech--Stone remainder of a locally compact Polish space.

Next, suppose that $A$ has an orthogonal sequence $r_i$ of projections
whose partial sums form an approximate unit.  Then
Conjecture~\ref{conj:ch} has been verified in two extreme cases.
First, if $r_i A r_j \neq \{0\}$ for all $i\neq j$, then it follows
from our methods in Section~\ref{S.Main}.  On the other hand, if $r_i
A r_j = \{0\}$ for all $i\neq j$, then $A=\bigoplus_j r_j A r_j$.  In
this situation the corona of $A$ turns out to be \emph{countably
  saturated} as a metric structure, and the conclusion of
Conjecture~\ref{conj:ch} follows from results of \cite{FaHa:Countable}
and \cite{BYBHU}.  However, it is not difficult to construct a
C*-algebra with an orthogonal sequence of projections whose partial
sums form an approximate unit for which neither method can be
applied.

Another case in which Conjecture~\ref{conj:ch} has been confirmed is
that of an algebra $A$ such that for every separable subalgebra $B$ of
$M(A)$, there is a $B$-quasicentral approximate unit for $A$
consisting of projections and the center of $\cq(A)$ is separable.
Here, an approximate unit $a_\lambda$ is \emph{$B$-quasicentral} if
$[b,a_\lambda]\to 0$ for every $b\in B$ (see
\cite[Corollary~2.14]{FaHa:Countable}).  The final case is in the
projectionless domain, where the topological reformulation was
confirmed in the case when $A=\cc_0([0,1))$ by a result of J.\ C.\ Yu
(see \cite[\S 9]{Hart:Cech}).

The problem of establishing Conjecture~\ref{conj:rigid} for coronas of
separable C*-algebras is much more interesting (and more
challenging!), and has only been verified in a few special cases.  The
case of $\cc(\beta\bbN\smallsetminus\bbN)$ was established by Shelah
in \cite{Sh:Proper}.  Once again, this result is only in hindsight:
Shelah was working in a Boolean-algebraic reformulation given by Stone
duality in the real rank zero case.  Following this, the second author
handled the case of several other abelian algebras in \cite{Fa:AQ},
and as we have mentioned, the Calkin algebra in \cite{Fa:All}.  Most
recently, McKenney addressed the case of products of UHF algebras
\cite{McK:UHF}. The analogues of 
Conjecture~\ref{conj:ch} and Conjecture~\ref{conj:rigid}
for Calkin algebras associated to  nonseparable Hilbert spaces
were considered in \cite{FaMcKSc:Some} and  \cite{Fa:AllAll}, respectively.

We now turn to an outline of the proof of Theorem~\ref{thm:main}, and
a discussion of how it is organized throughout the coming sections.
In Section~\ref{S.Build}, we construct an inverse system of abelian
groups, each of which consists of (equivalence classes of) elements of
the infinite torus $\TT^\NN$.  We then show, by building a complete
binary tree consisting of partial threads through this inverse system,
that the inverse limit will have many elements which are nontrivial in
the sense that they do not arise from constant threads.  This sort of
construction is familiar in category theory, and in
Section~\ref{S.Cohomology} we elaborate upon this connection.

Next, in Section~\ref{S.Main}, we construct a map from the inverse
limit built in Section~\ref{S.Build} into the automorphism group of
$\cq(A)$.  This is done by stratifying $\cq(A)$ into layers, and
identifying elements of the abelian groups of Section~\ref{S.Build}
with automorphisms of the layers.  In Section~\ref{S.Main} we also
isolate a technical condition on the C*-algebra $A$ to ensure that the
resulting map is one-to-one.  In Section~\ref{S.Weak}, we give a
second, weaker technical condition which again ensures that the map is
one-to-one.  Finally, in Section~\ref{S.Algebras}, we conclude the
proof of Theorem~\ref{thm:main} by verifying that each of its
alternative hypotheses (1)--(4) implies that one of the technical
assumptions of Sections~\ref{S.Main} or~\ref{S.Weak} is satisfied.

This argument is similar in one aspect to the proof in the case of the
Calkin algebra found in \cite{PhW}.  As in that proof, we end up with
a complete binary tree of height $\aleph_1$ consisting of partial
automorphisms of $\cq(A)$, and the branches of the tree determine
distinct automorphisms.  However, in \cite{PhW} the partial
automorphisms are defined on separable subalgebras, and most of the
difficulty in the argument lies in showing that they can be extended
at limit stages.  (It requires ensuring that the automorphisms are
asymptotically inner.)  In our approach, based on \cite{Fa:All}, we
stratify $\cq(A)$ into nonseparable layers. As we shall see, this
makes the limit stages much easier to handle.

Another difference with the proof in \cite{PhW} is that we will not
require the full strength of the Continuum Hypothesis (which in the
future we will abbreviate CH).  As a consequence, we can conclude that
the combinatorics of our proof are quite different from the
essentially model-theoretic methods used in Rudin's proof in the case
of $\cc(\beta\bbN\minus \bbN)$.  Indeed, 
in the forcing extension constructed in  \cite[Corollary~2]{FaSh:Trivial}
all automorphisms of $\cc(\beta\bbN\minus \bbN)$ are trivial, while
$\mathfrak d=\aleph_1$ and $2^{\aleph_0}<2^{\aleph_1}$.  And as we
shall see in the proof, the latter two assumptions suffice to
establish Theorem~\ref{thm:main} and its corollary.

Lastly, we would like to mention the important and related question of
whether there exists an automorphism of the Calkin algebra which is
$K$-theory reversing.  Unfortunately, our techniques cannot be used to
answer it, since the automorphisms we construct are locally trivial.
An answer to this question would require an extension of the
model-theoretic methods of~\cite{FaHa:Countable}.

We are indebted to N. Christopher Phillips for valuable conversations and 
we would like to acknowledge Paul McKenney for his feedback on an
early version of this paper.

\section{Building blocks for automorphisms}
\label{S.Build} 

Following \cite[\S1]{Fa:All}, we begin by constructing coherent
sequences (of uncountable length) of elements of the infinite torus
$\TT^\NN$.  In later sections, elements of $\TT^\NN$ will feature
prominently as the nonzero entries of diagonal unitary matrices, and
the coherent sequences constructed here will ultimately be patched
together to define automorphisms of C*-algebras.

To explain what is meant by ``coherent'', let us introduce a family of
pseudometrics $\Delta_I$ on $\TT^\NN$, where $I$ ranges over the
finite subsets of $\NN$.  The connection between $\Delta_I$ and
automorphisms of C*-algebras will become clear in Lemma~\ref{lem:key},
where it will be shown that each $\Delta_I$ corresponds to the
distance between two automorphisms when restricted to a certain
fragment of the corona algebra.

Initially, we define $\Delta_I$ when $I=\set{i,j}$ consists of just
two elements.  For $\alpha,\beta\in\TT^\NN$ we let
\[\Delta_{\set{i,j}}(\alpha,\beta)
=\abs{\alpha(i)\overline{\alpha(j)}-\beta(i)\overline{\beta(j)}}\;.
\]
Then, for $I$ a finite subset of $\NN$ we write
\[\Delta_I(\alpha,\beta)=\max_{i,j\in I}\Delta_{\set{i,j}}(\alpha,\beta)\;.
\]
It is clear that the $\Delta_I$ satisfy the triangle inequality
\[\Delta_I(\alpha,\gamma)\leq\Delta_I(\alpha,\beta)+\Delta_I(\beta,\gamma)\;,
\]
and therefore $\Delta_I$ is a pseudometric on $\TT^\NN$.  In most
cases, we will only need to evaluate $\Delta_I(\alpha,1)$, and this is
easily seen to be the diameter of the set of values of $\alpha$ on $I$.

We have the following inequality relating two different $\Delta_I$'s.

\begin{lemma} 
  \label{L.IJ}
  Let $I$ and $J$ be finite subsets of $\NN$.  Then for any fixed
  $i_0\in I$ and $j_0\in J$ we have
  \[\Delta_{I\cup J}(\alpha,\beta)
  \leq\Delta_I(\alpha,\beta)+\Delta_J(\alpha,\beta)
  +\Delta_{\set{i_0,j_0}}(\alpha,\beta)\;.
  \]
\end{lemma}

\begin{proof}
  We begin by observing that $\Delta_{\set{i,j}}(\alpha,\beta)$ can
  also be written
  \[\Delta_{\set{i,j}}(\alpha,\beta)
  =\abs{\alpha(i)\overline{\beta(i)}-\alpha(j)\overline{\beta(j)}}\;.
  \]
  It follows that for fixed $\alpha,\beta$, we have that
  $\Delta_{\set{\cdot,\cdot}}(\alpha,\beta)$ satisfies the following
  triangle inequality in the indices:
  \[\Delta_{\set{i,k}}(\alpha,\beta)
  \leq\Delta_{\set{i,j}}(\alpha,\beta)+\Delta_{\set{j,k}}(\alpha,\beta)\;.
  \]
  Thus, for $i\in I$ and $j\in J$, we have
  \[\Delta_{\set{i,j}}(\alpha,\beta)
  \leq\Delta_{\set{i,i_0}}(\alpha,\beta)
  +\Delta_{\set{i_0,j_0}}(\alpha,\beta)
  +\Delta_{\set{j_0,j}}(\alpha,\beta)\;,
  \]
  and the desired inequality follows.
\end{proof} 

For an infinite subset $X\subset\bbN$, we will need the following
notation.  First, for $j\in\NN$ let $n(X,j)$ denote the
$j^{\mathrm{th}}$ element in the increasing enumeration of $\{0\}\cup
X$.  Next, let
\[I(X,j)=\big[\,n(X,j), n(X,j+1)\,\big)
\]
denote the $j^{\mathrm{th}}$ interval in the natural partition of
$\bbN$ into finite intervals with endpoints in $X$.  We are now
prepared to make our key definition.

\begin{definition}\label{D.1.2}
  For $X\subset\NN$, we let $\bfF_X$ denote the subgroup of $\TT^\NN$
  defined by
  \[\bfF_X=\set{\alpha\in \cUN\;\vline\;
    \lim_{j\to \infty} \Delta_{I(X,j)\cup I(X,j+1)} (\alpha, 1)=0}\;,
  \]
  and let $\bfG_X$ denote the quotient
  \[\bfG_X=\cUN/\bfF_X\;.
  \]
  (We will consider these groups as discrete.)
\end{definition}

As we shall see in Lemma~\ref{lem:welldefined}, the elements of
$\bfF_X$ will give rise to automorphisms of a corona algebra which are
trivial on a certain fragment of that algebra.
For future convenience, we presently note that Lemma~\ref{L.IJ}
implies that $\bfF_X$ can also be written as
\[\bfF_X=\set{\alpha\in\TT^\NN\;\vline\;
  \begin{aligned}
    &\lim_{j\to \infty}\Delta_{I(X,j)}(\alpha, 1)=0\text{, and}\\
    &\lim_{j\to \infty}\Delta_{\set{n(X,j),n(X,j+1)}}(\alpha,1)=0
  \end{aligned}
}\;.
\]

Note that if $Y\subset X$ then every $I(Y,j)$ can be written as
$\bigcup_{k\in L}I(X,k)$ for some finite set $L$.  Since we have
$\Delta_{I(Y,j)}\geq \max_{k\in L} \Delta_{I(X,k)}$, it follows that
$\bfF_Y\subset\bfF_X $.  Moreover, if the symmetric difference of $Y$
and $Z$ is finite then $\bfF_Y=\bfF_Z$.  Therefore, we have the
following result.

\begin{proposition}
  \label{prop:incl}
  If $Y\subset^*X$ then $\bfF_Y\subset\bfF_X $.  Hence, also
  $\bfG_X$ is a quotient of $\bfG_Y$.
\end{proposition}

(Here, $\subset^*$ denotes the \emph{almost inclusion} relation, that
is, $Y\subset^*X$ if and only if $Y\minus X$ is finite.)

What is more, it is easy to see that for $Y\subset^*X$ we have the
following commutative diagram whose rows are exact sequences.
\begin{equation}
\label{I.short}
\begin{diagram}[h=1cm,w=1cm]
 0 & \rTo & \bfF_Y & \rTo & \cUN & \rTo & \bfG_Y & \rTo & 0 \\
 & & \dTo & & \dTo & & \dTo \\ 
 0 & \rTo & \bfF_X & \rTo & \cUN & \rTo & \bfG_X & \rTo & 0
\end{diagram}
\end{equation}
Here, the arrow $\bfF_Y\rTo\bfF_X$ is the inclusion map, the arrow
$\TT^\NN\rTo\TT^\NN$ is the identity, and the arrow $\bfG_Y\rTo\bfG_X$
is the quotient map.

For the remainder of this section, we will work with a family
$\cU\subset\cP(\bbN)$ consisting of infinite sets, which we will
assume has the following properties.

\begin{hypothesis}\
  \label{hyp:U}
  \begin{itemize}
  \item $\cU$ is closed under finite intersections and under finite
    modifications of its elements.
  \item $\cU$ is $\aleph_1$-generated, that is, there exist
    $X_\xi\in\cU$, $\xi<\omega_1$, such that for any $X\in\cU$ there
    exists $\xi$ with $X_\xi\subset^* X$.
  \item The family of enumerating functions $n(X,\cdot)$ of elements
    $X\in\cU$ forms a \emph{dominating family}.  That is, for any
    $f\in\NN^\NN$ there exists $X\in\cU$ such that $f(i)\leq n(X,i)$
    for all but finitely many $i\in \bbN$.
  \end{itemize}
\end{hypothesis}

We remark that the assumption that such a family $\cU$ exists follows
from CH.  In fact, it follows from the axiom $\mathfrak d=\aleph_1$,
which means: the least cardinality of a dominating family is exactly
$\aleph_1$.  This axiom is strictly weaker than CH---see \cite[p.\
629]{Fa:All} for a discussion.

We are now ready to present the main technical result of this section.
Recall that by Proposition~\ref{prop:incl} and the discussion
following it, the groups $\bfG_X$, for $X\in \cU$, form an inverse
system indexed by $\cU$ with respect to the reverse almost inclusion
ordering $\supset^*$.  The main result is simply to count the elements
of $\ilim_{X\in\cU}\bfG_X$.

\begin{theorem}
  \label{thm:ilim}
  If $\cU$ satisfies Hypothesis~\ref{hyp:U}, then
  $\ilim_{X\in\cU}\bfG_X$ has cardinality $2^{\aleph_1}$.
\end{theorem} 

As a consequence, CH implies that most of the elements of
$\ilim_{X\in\cU}\bfG_X$ are \emph{nontrivial}, in the sense that they
do not arise from a constant thread.

\begin{corollary}
  \label{cor:nontrivial}
  Suppose that CH holds, and that $\cU$ satisfies
  Hypothesis~\ref{hyp:U}.  Then the quotient
  $\ilim_{X\in\cU}\bfG_X/\im(\TT^\NN)$ is nontrivial, where
  $\im(\TT^\NN)$ denotes the image of $\TT^\NN$ under the diagonal map
  $\alpha\mapsto([\alpha]_{\bfF_X})_{X\in\cU}$.
\end{corollary}

We remark that to establish the Corollary, it is enough to replace the
assumption of CH with the so called \emph{weak Continuum
  Hypothesis}---the statement that $2^{\aleph_0}<2^{\aleph_1}$.

In later sections, we will show how to construct automorphisms of a
corona algebra from elements of the inverse limit
$\ilim_{X\in\cU}\bfG_X$.  There, we will use Theorem~\ref{thm:ilim} it
to show the analog of Corollary~\ref{cor:nontrivial} that most of the
automorphisms we construct will be \emph{nontrivial automorphisms}.

\begin{proof}[Proof of Theorem~\ref{thm:ilim}]
  Fix a strictly $\subset^*$-decreasing chain $X_\xi$, for
  $\xi<\omega_1$ that generates $\cU$.  We must construct
  $\alpha_s\in\TT^\NN$, for $s\in 2^\xi$ and $\xi<\omega_1$, by
  recursion so that (writing $\bfF_\xi$ for $\bfF_{X_\xi}$)
  \begin{itemize}
  \item $s\sqsubset t$ implies $\alpha_s\alpha_t^{-1}\in \bfF_\xi$,
    where $\xi=\dom(s)$, and
  \item $\alpha_{s^\frown 0}\alpha_{s^\frown 1}^{-1}\notin
    \bfF_{\xi+1}$, where $\xi=\dom(s)$.
  \end{itemize}

  For the successor stage, it suffices to show that
  $\bfF_{\xi+1}\subsetneq\bfF_\xi$.  For this, we initially thin out
  the sequence $X_\xi$ to suppose that for all $m$ there exists $n(m)$
  such that at least $m$ intervals of $X_\xi$ are contained in the
  interval $I(X_{\xi+1},n(m))$.  This can be done without any loss of
  generality by our third assumption in Hypothesis~\ref{hyp:U}
  (although it would also suffice to assume that $\cU$ is a
  nonprincipal ultrafilter).  Now, it is easy to construct an element
  $\alpha\in\TT^\NN$ which is constant on each $I(X_\xi,j)$, satisfies
  $\alpha(n(X_\xi,j+1))=e^{i\pi/m}\alpha(n(X_\xi,j))$ whenever
  $I(X_\xi,j)\subset I(X_{\xi+1},n(m))$, and satisfies
  $\alpha(n(X_\xi,j+1))=\alpha(n(X_\xi,j))$ otherwise.  Clearly
  $\alpha\in\bfF_\xi$, but for each $m$ we have
  $\Delta_{I(X_{\xi+1},n(m))}(\alpha,1)=2$ and so
  $\alpha\notin\bfF_{\xi+1}$.

  Next, we suppose that $s\in2^\xi$ and $s=\bigcup s_n$ where
  $\alpha_{s_n}$ have been defined for all $n$.  Since $s_n$ and $s$
  are the only sequences we are interested in, we may simplify the
  notation and write $\alpha_n$, $X_n$, and $\bfF_n$ for
  $\alpha_{s_n}$, $X_{\dom(s_n)}$, and $\bfF_{X_{\dom(s_n)}}$ respectively.
  We will choose $X_\infty$ to be a subset of $X_\xi$ which is so
  sparse that it satisfies
  \begin{eqtext}
    \label{e:xinfty}
    whenever $n\leq k$ and $I(X_n,j)\subset I(X_\infty,k)$, we have\\
    $\Delta_{I(X_n,j)}(\alpha_k,\alpha_n)<1/k$.
  \end{eqtext}
  For this, we inductively choose $I(X_\infty,k-1)$ large enough to
  include all of the remaining $I(X_n,j)$ such that $n\leq k$ and
  $\Delta_{I(X_n,j)}(\alpha_k,\alpha_n)\geq1/k$.  This can be done
  because for each fixed $n\leq k$, our first bulleted inductive
  hypothesis implies that
  $\Delta_{I(X_n,j)}(\alpha_k,\alpha_n)\rightarrow0$.  Using exactly
  the same reasoning, we can also suppose that
  \begin{eqtext}
    \label{e:xinfty2}
    whenever $n\leq k$ and $I(X_n,j)\subset I(X_\infty,k)$, we have\\
    $\Delta_{\set{n(X_n,j),n(X_n,j+1)}}(\alpha_k,\alpha_n)$ $<1/k$.
  \end{eqtext}

  Next, we must define $\alpha_s$ so that for each $n$, we have
  $\alpha_s\alpha_n^{-1}\in\bfF_n$.  In other words, we will need to
  satisfy both:
  \begin{equation}
    \label{e:limitdelta}
    \lim_{j\rightarrow\infty}\Delta_{I(X_n,j)}(\alpha_s,\alpha_n)=0\\
  \end{equation}
  \begin{equation}
    \label{e:limitrho}
    \lim_{j\to \infty}\Delta_{\set{n(X_n,j),n(X_n,j+1)}}(\alpha_s,\alpha_n)=0
  \end{equation}
  For \eqref{e:limitdelta} it would be sufficient to let
  $\alpha_s(i)=\alpha_n(i)$ whenever $i\in I(X_\infty,n)$.  However,
  to establish \eqref{e:limitrho} we shall need to be a little more
  careful.  Specifically, we define $\alpha_s(i)$ inductively such
  that for all $i\in I(X_\infty,n)\cup\{n(X_\infty,n+1)\}$ we have
  $\alpha_s(i)=\gamma_n\,\alpha_n(i)$, where $\gamma_n\in\TT$ is a
  uniquely determined constant.

  Now to verify \eqref{e:limitdelta}, let $k(j)$ be such that $n\leq
  k(j)$ and $I(X_n,j)\subset I(X_\infty,k(j))$.  Then using the
  definition of $\alpha_s$, together with the fact that constant
  multiples do not have an effect on the value of $\Delta_I$, we see
  that
  \begin{equation*}
    \Delta_{I(X_n,j)}(\alpha_s,\alpha_n)
    =\Delta_{I(X_n,j)}(\gamma_{k(j)}\alpha_{k(j)},\alpha_n)
    =\Delta_{I(X_n,j)}(\alpha_{k(j)},\alpha_n)\;.
  \end{equation*}
  By \eqref{e:xinfty}, the latter term is $<1/k(j)$.
  Since $k(j)\rightarrow\infty$ as $j\rightarrow\infty$, this
  establishes \eqref{e:limitdelta}.  But now \eqref{e:limitrho} is
  similar, because using the same reasoning we have
  \[\Delta_{\set{n(X_n,j),n(X_n,j+1)}}(\alpha_s,\alpha_n)
  =\Delta_{\set{n(X_n,j),n(X_n,j+1)}}(\alpha_{k(j)},\alpha_n)\;,
  \]
  and the latter term is $<1/k(j)$ by property \eqref{e:xinfty2}.
\end{proof}


\section{Connection with cohomology}
\label{S.Cohomology}

In this section we explore a connection between
Corollary~\ref{cor:nontrivial} and cohomology theory.  Note that a
similar connection exists in the work of the second author (see the
discussion in \cite[\S 2]{Fa:AQ}) as well as Talayco \cite{talayco}.
We assume the reader is familiar with the most basic categorical
notions, such as short exact sequences.

We begin by observing that the conclusion of the previous section,
that $\ilim\bfG_X/\im(\TT^\NN)$ can be nontrivial, stems from the fact
that surjective maps need not remain so after passing to the inverse
limit.  Indeed, notice that each of the quotient maps $\TT^\NN\to
\bfG_X$ is surjective, but Theorem~\ref{thm:ilim} implies that the
natural map $\TT^\NN\to\ilim\bfG_X$ is not surjective.

In the language of category theory, we say that $\ilim$ is a
left-exact covariant functor which is not right-exact.  This means
that whenever we apply it to a short exact sequence of objects
\[0\rTo\cA\rTo\cB\rTo\cC\rTo0
\]
we obtain an exact sequence
\[0\rTo\ilim\cA\rTo\ilim\cB\rTo\ilim\cC
\]
but we can not in general add the last $\rTo0$.  As we shall explain,
it follows that $\ilim$ gives rise to a sequence of derived functors
in the same way that $\Hom(A,\cdot)$ gives rise to the functors
$\Ext^n(A,\cdot)$.  The derived functors of $\ilim$ are denoted by
$\ilim^n$.  We refer the reader to \cite{Jensen} for a detailed
introduction to the derived functors $\ilim^n$.  What we will show is
that the conclusion of Corollary~\ref{cor:nontrivial} corresponds to the
statement that the first derived inverse limit $\ilim^1\bfF_X$ is
nontrivial.

Proceeding generally, let $D$ be a directed set and work in the
category of inverse systems of abelian groups which are indexed by
$D$.  In other words, an object is a sequence $\mathcal A=(A_d:d\in
D)$ together with a system of maps $\pi_{de}\colon A_e\to A_d$ for
$d<e$, satisfying the usual composition law: if $d<e<f$ then
$\pi_{de}\circ\pi_{ef}=\pi_{df}$.  Given such an $\mathcal A$, it is
possible to find an \emph{injective resolution}, that is, an exact
sequence of injective objects:
\begin{equation}
  \label{e:inj}
  0\rTo \mathcal A\rTo Q^0(\mathcal A)\rTo Q^1(\mathcal A)\rTo\cdots\;.
\end{equation}
Although we have no need for the details, we can find such a
resolution explicitly as follows.  First, for each $A_d$ let $M_d$ be
an injective abelian group containing $A_d$.  Then define
$Q^0(\mathcal A)$ to be the inverse system $(Q_d:d\in D)$ where
$Q_d=\prod_{d'\leq d}M_{d'}$, with respect to the natural restriction
maps.  It is not hard to check that $Q^0(\mathcal A)$ is an injective
object, and that $\mathcal A$ embeds into $Q^0(\mathcal A)$.  One then
continues the resolution inductively by letting $Q^n(\mathcal A)$ be
the $Q^0$ of the cokernel of the previous map.

We now we apply $\ilim$ to each term in the resolution in
Equation~\eqref{e:inj} to obtain a sequence:
\begin{equation}
  \label{e:cochain}
  0\rTo\ilim\mathcal A\rTo\ilim Q^0(\mathcal A)
  \rTo\ilim Q^1(\mathcal A)\rTo\cdots\;.
\end{equation}
Let us denote this sequence as a whole by $Q(\cA)$, and each map
$\ilim Q^n(\cA)\rightarrow\ilim Q^{n+1}(\cA)$ by $d^n$.  Then $Q(\cA)$
is not necessarily exact, but it still has the property that $\im
d^n\subset\ker d^{n+1}$.  Such a sequence is called a \emph{cochain},
and the maps $d^n$ are called \emph{coboundary} maps.  The
\emph{cohomology} groups $H^n(Q(\cA))=\ker d^{n+1}/\im d^n$ measure
the inexactness of the cochain (here, $H^0(Q(\cA))=\ker d^0$).

\begin{definition}
  \label{def:derived}
  For each $n$, the derived functor $\ilim^n$ is defined by
  \[\ilim^n\mathcal A=H^n(Q(\mathcal A))\;.
  \]
\end{definition}

It is a standard fact that this definition is independent of the
choice of injective resolution $Q^n(\mathcal A)$.

Since $\ilim$ is left-exact, the cochain in Equation~\eqref{e:cochain}
is exact at the term $\ilim Q^0(\mathcal A)$.  It follows that
$\ilim^0\mathcal A$ is precisely $\ilim\cA$; in other words,
the functor $\ilim^0$ is precisely $\ilim$.  In many cases the higher
$\ilim^n$ groups are trivial, and a number of conditions have been
established to guarantee that $\ilim^1\cA$ vanishes.  In this paper we
only have need of the following.

\begin{definition}
  An inverse system $\mathcal A$ is said to be \emph{flasque} if for
  every downwards-closed $J\subset D$, every partial thread
  $(a_d)_{d\in J}$ can be extended to a thread $(a_d)_{d\in D}$.
\end{definition}

It is a basic result that if $\mathcal A$ is flasque, then
$\ilim^n\mathcal A=0$ for all $n\geq1$ (see
\cite[Th\'eor\`eme~1.8]{Jensen}).  Flasque systems are not uncommon;
for instance if $\mathcal A=(A_i)_{i\in\NN}$, and the maps $\pi_{ij}$
are all surjective, then $\mathcal A$ is flasque.

We now turn to a computation of the higher $\ilim^n$ groups for the
inverse systems from the previous section.  Let $\mathcal F$ denote
the inverse system consisting of $(\bfF_X)_{X\in \cU}$ with respect to
the inclusion maps, and similarly let $\mathcal G$ denote
$(\bfG_X)_{X\in \cU}$ with respect to the quotient maps.  Notice that
$\mathcal F$ is rather degenerate as an inverse system; its
``projection'' maps are injective!  Nevertheless, we now show that
it's first derived inverse limit $\ilim^1\mathcal F$ is nontrivial,
and in fact it coincides precisely with the object of study of
Corollary~\ref{cor:nontrivial}.

\begin{proposition}
  \label{prop:lim1}
  The group $\ilim^1\mathcal F$ is precisely $\ilim\mathcal
  G/\im(\TT^\NN)$, where $\TT^\NN\to\ilim\mathcal G$ is the diagonal
  map.
\end{proposition}

\begin{proof}
  Rather than proceed via a direct computation (which would involve
  the messy injective resolution), we will use the fact that
  $\ilim^1\mathcal F$ appears in a long exact sequence.  Let $\mathcal
  T$ denote the inverse system consisting of $(\TT^\NN)_{X\in\mathcal
    U}$ together with the identity maps.  Then we see from
  Equation~\eqref{I.short} that in the category of inverse systems, we
  have an exact sequence:
  \[0\rTo\mathcal F\rTo\mathcal T\rTo\mathcal G\rTo 0\;.
  \]
  Again letting $Q(\mathcal A)$ denote the cochain constructed in
  Equation~\eqref{e:cochain}, it follows that there is an exact sequence
  of cochains
  \[0\rTo Q(\mathcal F)\rTo Q(\mathcal T)\rTo Q(\mathcal G)\rTo 0\;.
  \]
  By the zig-zag lemma, we get a long exact sequence consisting of the
  cohomology groups of $Q(\mathcal F)$, $Q(\mathcal T)$, and
  $Q(\mathcal G)$.  By Definition~\ref{def:derived} and the following
  remark, this means that we have an exact sequence:
  \begin{align*}
    0\rTo\ilim\mathcal F\rTo\ilim\mathcal T\rTo&\ilim\mathcal G
    \rTo\ilim^1\mathcal F\rTo \ilim^1\mathcal T\rTo\cdots\;.
    \intertext{Now, it is clear that that $\ilim T$ is just $\TT^\NN$.
      Second, it is easy to see that $\mathcal T$ is flasque, and
      hence that $\ilim^1(\mathcal T)=0$.  Hence, we have an exact
      sequence:}
    \quad\quad\;\TT^\NN\rTo&\ilim\mathcal G\rTo\ilim^1\mathcal F\rTo 0\;.
  \end{align*}
  But this precisely means that $\ilim^1\mathcal F$ is isomorphic to
  $\ilim\mathcal G/\im(\TT^\NN)$, as desired.
\end{proof}

For completeness, it is worth showing that the object identified in
Proposition~\ref{prop:lim1} is the last interesting one in the
picture.

\begin{proposition}
  The derived inverse limits $\ilim^n\mathcal F$ vanish for all $n>1$.
\end{proposition}

\begin{proof}
  We shall use the fact \cite[\S 2]{Jensen} that if $J$ is cofinal in
  $\cU$, then $\ilim^n\mathcal F=\ilim_J^n\mathcal F$ for all $n$.
  Let $J=\set{X_\xi:\xi<\omega_1}$ denote the subset of $\cU$
  consisting just of the elements of the generating tower.  Then it is
  again clear that $\mathcal T\rs J$ is flasque, but moreover
  $\mathcal G\rs J$ is flasque as well.  Indeed, if $X_\xi\in J$ then
  any partial thread $(g_{X_\alpha})_{\alpha<\xi}$ can be extended to
  a thread on all of $J$ using the construction at the limit stage in
  the proof of Theorem~\ref{thm:ilim}.  Therefore in the exact
  sequence
  \begin{equation*}
    0\rTo\ilim_J\mathcal F\rTo\ilim_J\mathcal T\rTo\ilim_J\mathcal G
    \rTo\ilim_J^1\mathcal F\rTo \ilim_J^1\mathcal T\rTo\cdots
  \end{equation*}
  we have that two out of every three terms beginning with
  $\ilim_J^1\mathcal T$ is equal to $0$.  It follows that the
  sequence vanishes at $\ilim_J^1\mathcal T$, and in particular
  $\ilim^n\mathcal F=\ilim_J^n\mathcal F=0$ for $n>1$.
\end{proof}

\section{The main result}
\label{S.Main} 

In this section we show how elements of the group $\ilim\bfG_X$
constructed in Section~\ref{S.Build} give rise to automorphisms of the
corona algebra of a C*-algebra $A$.  For this, we shall need to impose
the following hypothesis on $A$.

\begin{hypothesis}
  \label{hyp:A}
  We assume $A$ has a sequence of orthogonal projections $r_i$, for
  $i\in \bbN$, such that:
  \begin{itemize}
  \item $r_i A r_j\neq 0$ for all $i$ and $j$; and 
  \item the sequence of partial sums $p_n=\sum_{i<n}r_i$ form an
    \emph{approximate unit} for $A$.  That is, for any $a\in A$ we
    have $p_na\rightarrow a$.
  \end{itemize}
\end{hypothesis}

The following is our first generalization of the result of Phillips
and Weaver.  In the next section, we shall strengthen it by slightly
weakening Hypothesis~\ref{hyp:A}.

\begin{theorem}
  \label{thm:corona}
  Suppose that $\cU\subset\mathcal P(\NN)$ satisfies
  Hypothesis~\ref{hyp:U} and that $A$ satisfies
  Hypothesis~\ref{hyp:A}.  Then there is an embedding of
  $\ilim_{X\in\cU}\bfG_X$ into the automorphism group of $\cq(A)$.
\end{theorem}

The proof will consist of a series of lemmas.  We begin by showing how
to stratify the corona $\cq(A)$ into a family of systems of operators
$C_X(A)$ indexed by elements $X\in\mathcal P(\NN)$.  Each system
$C_X(A)$ will consist of operators which are very near to being block
diagonal, in the sense of the $r_i$.  (Of course, we cannot hope to
stratify the corona using multipliers which are actually block
diagonal in this sense.)



For $I\subset\bbN$ let us define
\[
p_I=\sum_{i\in I} r_i\;.
\]
We then let
\[DD_X(A)=\set{m\in M(A):p_{I(X,i)}\,m\,p_{I(X,j)}=0\textrm{ whenever
  }\abs{i-j}\geq2}\;.
\]
We shall see during the proof of Lemma~\ref{lem:stratify} that every
element $m\in DD_X(A)$ can be written as a sum of two multipliers
which are block-diagonal in the sense of the $r_i$.

Now, we let $C_X(A)\subset \cq(A)$ be the fragment which comes from the
image of $DD_X(A)$ inside the corona:
\[
C_X(A)=DD_X(A)/(A\cap DD_X(A))\;.
\]
The following lemma shows that these fragments do indeed stratify
$\cq(A)$.  Special cases of it appeared during the initial segment of
the proof of \cite[Theorem 3.1]{Ell:Derivations} (where it was assumed
that $A$ is a UHF algebra) and \cite[Lemma 1.2]{Fa:All} (where
$A=\cK$).

\begin{lemma}
  \label{lem:stratify}
  For every $m\in M(A)$ there is a subset $X\subset\bbN$ such that $m$
  is represented in $C_X(A)$.  More precisely, every $m\in M(A)$ can
  be written as $m=d+a$ with $d\in DD_X(A)$ and $a\in A$.
\end{lemma} 

\begin{proof}
  Let $n(0)=0$ and inductively define an increasing sequence $n(j)$ so
  that the following holds for $j\geq1$:
  \begin{eqtext}
    $(1-p_{n(j+1)})\,m\,p_{n(j)} $ and $(1-p_{n(j+1)})\,m^*\,p_{n(j)}$
    have norm $\leq 2^{-j}$.
  \end{eqtext}
  This is possible since $m\,p_{n(j)}\in A$ and $p_n$ form an
  approximate unit for $A$.  We claim that if $X=\set{n(j):j\in\NN}$
  then $X$ is as desired.

  For this, we let $m_e$ and $m_o$ be defined as follows:
  \begin{align*}
    m_e &= \sum_i p_{I(X,2i)\cup I(X,2i+1)}\,m\,p_{I(X,2i)\cup I(X,2i+1)}\;,\\
    m_o &= \sum_i p_{I(2i+1)}\,m\,p_{I(2i+2)}+p_{I(2i+2)}\,m\,p_{I(2i+1)}\;.
  \end{align*}
  See Figure~\ref{fig:dd} for a clearer picture of how $m_e$ and $m_o$
  are selected.

  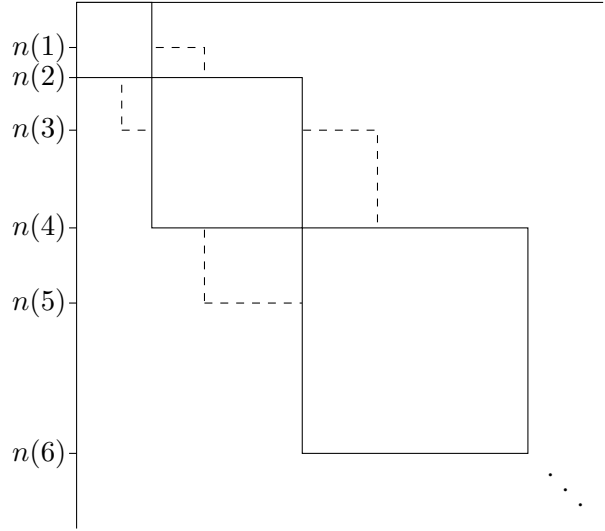
\begin{figure}[h]
  \begin{tikzpicture}
    \draw (0,0) -- (7,0);
    \draw (0,0) -- (0,-7);
    \draw[dashed] (.6,-.6) rectangle (1.7,-1.7);
    \draw[dashed] (1.7,-1.7) rectangle (4,-4);
    \draw[fill=white,draw=none] (0,0) rectangle (1,-1);
    \draw (0,0) rectangle (1,-1);
    \draw[fill=white,draw=none] (1,-1) rectangle (3,-3);
    \draw (1,-1) rectangle (3,-3);
    \draw[fill=white,draw=none] (3,-3) rectangle (6,-6);
    \draw (3,-3) rectangle (6,-6);
    \foreach \x / \y in {1/.6,2/1,3/1.7,4/3,5/4,6/6}
    {
    \node[anchor=east] at (0,-\y) {$n(\x)$};
    \draw (-.1,-\y) -- (0,-\y);
    }
    \foreach \z in {.3,.5,.7}
    \node at (6+\z,-6-\z) {$\cdot$};
  \end{tikzpicture}
  \caption{The solid square regions represent fragments of $m$
    captured in $m_e$; the dashed rectangular additions represent
    fragments of $m$ captured in $m_o$.  The details of the
    construction imply that the uncaptured region is summable in norm
    and therefore represents an element of $A$.\label{fig:dd}}
  \end{figure}

  Now, $d=m_e+m_o$ makes up a very large portion of $m$; indeed, a
  straightforward computation shows that $a=m-m_e-m_o$ satisfies
  \[\norm{(1-p_{n(i)})\,a}\leq 2^{-i+4}
  \]
  for all $i$ and therefore $a\in A$, as required.

  Finally, we need only verify that $d=m_e+m_o$ lies in $DD_X(A)$.
  Indeed, considering each term in the definitions of both $m_e$ and
  $m_o$, if $\abs{i-j}\geq2$ then the term is annihilated either by
  left-multiplication by $p_{I(X,i)}$ or right-multiplication by
  $p_{I(X,j)}$.
\end{proof}

We note that if $\cU\subset\cP(\bbN)$ is such that the enumerating
functions of elements of $\cU$ form an eventually dominating family,
then in Lemma~\ref{lem:stratify} we may choose $X$ to be an element of
$\cU$.

We now show how to define a continuous homomorphism from $\TT^\NN$
into the group of inner automorphisms of the special fragments
$C_X(A)$ of the corona defined above.  For this we will need the
following key result which connects the construction in
Section~\ref{S.Build} with our present efforts.

\begin{lemma}
  \label{lem:key}
  Suppose that $A_0$ is a C*-algebra which contains a sequence of
  orthogonal projections $q_0,\ldots,q_{n-1}$ satisfying
  \begin{itemize}
  \item $\sum_{i<n} q_i=1$, and 
  \item $q_i\,A_0\,q_j\neq \{0\}$ for all $i,j<n$. 
  \end{itemize}
  Letting $I=\set{0,\ldots,n-1}$, for each $\alpha\in\TT^I$ we let
  $u_\alpha=\sum_{i<n}\alpha(i) q_i$.  Then for all $\alpha,\beta$ we
  have
  \[\Delta_I(\alpha,\beta)\leq \norm{\Ad u_\alpha - \Ad
    u_\beta}\leq2\Delta_I(\alpha,\beta)\;.
  \]
\end{lemma}
  
\begin{proof} 
  Since $\Delta_I(\alpha,\beta)=\Delta_I(\alpha\bar \beta,1)$ and
  $\norm{\Ad u_\alpha - \Ad u_\beta}=\norm{\Ad (u_\alpha
    u_\beta^*)-\id}$ we may assume $\beta=1$ (i.e., the constantly 1
  function in $\TT^n$).  Now, to show that $\norm{\Ad
    u_\alpha-\id}\leq2\Delta_I(\alpha,1)$, first note that we can
  multiply $\alpha$ by a fixed constant to assume that $\alpha(0)=1$.
  This implies that for all $i<n$,
  $\abs{\alpha(i)-1}\leq\Delta(\alpha,1)$, and so in particular
  \[u_\alpha-1=\sum_{i<n}(\alpha(i)-1)q_i
  \]
  has norm $\leq\Delta(\alpha,1)$.  It follows that
  \begin{align*}
    \norm{u_\alpha\,a\,u_\alpha^*-a}&\leq
    \norm{u_\alpha\,a\,u_\alpha^*-u_\alpha\,a}+\norm{u_\alpha\,a - a}\\
    &\leq2\Delta(\alpha,1)\norm{a}
  \end{align*}
  
  For the inequality $\Delta_I(\alpha,1)\leq\left\|\Ad
    u_\alpha-\id\right\|$, fix $i$ and $j$ in $I$.  For $a\in
  q_i\,A_0\,q_j$ we have
  \[u_\alpha\,a\,u_\alpha^* - a = \alpha(i)\,q_i\,a\,q_j
  \overline{\alpha(j)}-q_i\,a\,q_j =(\alpha(i)\overline{\alpha(j)}-1)a
  \]
  and therefore
  \[\norm{\Ad u_\alpha-\id}\geq
  \max_{i,j\in I}\Delta_{\set{i,j}}(\alpha,1)=\Delta_I(\alpha,1)\;,
  \]
  as desired.
\end{proof} 

Now, for $\alpha\in \cUN$ define a unitary element of $M(A)$ by
\[u_\alpha=\sum_{i\in \bbN} \alpha(i)r_i\;,
\]
and let $\dot u_\alpha$ denote the image of $u_\alpha$ in $\cq(A)$.

\begin{lemma}
  \label{lem:welldefined}
  Suppose once again that $A$ satisfies Hypothesis~\ref{hyp:A}.  Given
  an infinite subset $X\subset\bbN$, the map $\cUN\to\Aut(C_X(A))$
  defined by
  \[\alpha\mapsto\Ad\dot u_\alpha
  \]
  has kernel precisely equal to $\bfF_X$.
\end{lemma}

\begin{proof}
  Fix $\alpha\in \bbT^{\bbN}$.  Given a multiplier $m\in DD_{X_e}(A)$,
  the proof of Lemma~\ref{lem:stratify} shows that $m$ can be written
  as $m_e+m_o$ where
  \begin{align*}
    m_e&\in\sum p_{I(X,2i)\cup I(X,2i+1)}\,A\,p_{I(X,2i)\cup I(X,2i+1)}\;,\\
    m_o&\in\sum p_{I(X,2i+1)\cup I(X,2i+2)}\,A\,p_{I(X,2i+1)\cup I(X,2i+2)}\;.
  \end{align*}
  Dealing first with $m_e$, it is not hard to compute that
  \[\norm{\dot u_\alpha\,\dot m_e\,\dot u_\alpha^*-\dot m_e}
  =\limsup_j \norm{(u_\alpha\,m_e\,u_\alpha^*-m_e)p_{I(X,2j)\cup I(X,2j+1)}}\;.
  \]
  Using Lemma~\ref{lem:key}, whenever $\alpha\in \bfF_X$ we have that
  the right-hand side is zero.  The same argument shows that we have
  $\norm{\dot u_\alpha\,\dot m_o\,\dot u_\alpha^*-\dot m_o}=0$, and it
  follows that $\alpha$ is in the kernel.
 
  Conversely, if $\alpha\notin \bfF_X$ then there is $\epsilon>0$ such
  that for infinitely many $i\in\NN$ we have
  \[\Delta_{I(X,i)\cup I(X,i+1)}(\alpha,1)\geq \epsilon. 
  \]
  By Lemma~\ref{lem:key}, for such $j$ we can find an element
  \[a_i\in p_{I(X,i)\cup I(X,i+1)}\,A\,p_{I(X,i)\cup I(X,i+1)}
  \]
  such that $\norm{u_\alpha\,a_i\,u_\alpha^*-a_i}\geq
  \epsilon\norm{a_i}$.  We may assume either that all of the $i$'s are
  even or that all of the $i$'s are odd.  Renormalizing so that
  $\|a_i\|=1$ for all such $i$ and letting $a=\sum_i a_i$, we have an
  element of $DD_X(A)$ (in fact a block-diagonal one) witnessing that
  $\Ad \dot u_\alpha$ is not the identity map.
\end{proof} 

The final component of the argument involves patching together a
coherent sequence of partial inner automorphisms to construct an outer
automorphism.  Suppose that $\cU\subset\cP(\bbN)$ is directed under
$\supset^*$.  For $X\in \cU$ let
\[\bfA_X=\{\Phi\in \Aut(C_X(A))\mid\Phi\text{ normalizes }C_Y(A)
\text{ whenever }X\subset^* Y\in \cU\}\;,
\]
and for $X\subset^*Y$ let $\pi_{YX}\colon \bfA_X\to \bfA_Y$ denote
the map
\[
\pi_{YX}(\Phi)=\Phi\rs C_Y(A)\;.
\]
Then $(\bfA_X)_{X\in\cU}$ together with the maps $\pi^X_Y$ forms an
inverse system of groups.  The following result shows that, assuming
$\cU$ is large enough, any thread through $(\bfA_X)_{X\in\cU}$ gives
rise to an element of $\Aut(\cq(A))$.  This is the evolution of
\cite[Lemma~1.3]{Fa:All}.
   
\begin{lemma}
  \label{lem:patch} 
  Suppose that $A$ satisfies Hypothesis~\ref{hyp:A} and
  $\cU\subset\mathcal P(\NN)$ satisfies Hypothesis~\ref{hyp:U}.  Then
  there is a group homomorphism from $\ilim\bfA_X\to\Aut(\cq(A))$ of
  the form
  \[ (\Phi_X)_{X\in \cU}\mapsto \Phi
  \]
  which satisfies $\Phi(\dot m)=\Phi_X(\dot m)$ for $\dot m\in
  C_X(A)$.
\end{lemma} 

\begin{proof}
  Given $\Phi_X$ for $X\in\cU$ and $b\in M(A)$ we define $\Phi(\dot
  m)$ as follows.  By Lemma~\ref{lem:stratify} and the remark
  following its proof there exists $X\in \cU$ such that $m=d+a$ where
  $d\in DD_X(A)$ and $a\in A$.  Then we simply let:
  \[\Phi(\dot m)=\Phi_X(\dot d)\;.
  \]
  Then since $X\supset^*Y$ implies that $\Phi_Y\rs \bfA_X$ agrees with
  $\Phi_X$, we clearly have that $\Phi(\dot m)$ does not depend on the
  choice of $X$.  Moreover, $\Phi$ is invertible since its inverse
  corresponds to the thread $\left(\Phi_X^{-1}\right)_{X\in\cU}$.
\end{proof} 

Putting together the maps from Lemma~\ref{lem:welldefined} and
Lemma~\ref{lem:patch} we now obtain the desired embedding.

\begin{proof}[Proof of Theorem~\ref{thm:corona}]
  Let $(\alpha_X)_{X\in\cU}$ be given representative of
  $\ilim_{X\in\cU}\bfG_X$ (in other words, the residues
  $([\alpha_X]_{\bfF_X]})_{X\in\cU}$ form a thread in the inverse
  system $(\bfG_X)_{X\in\cU}$).  Then it is easy to see that $(\Ad\dot
  u_{\alpha_X})_{X\in\cU}$ form a thread in $(\bfA_X)_{X\in\cU}$, and
  the map $\ilim_{X\in\cU}\bfG_X\to\ilim_{X\in\cU}\bfA_X$
  defined by
  \[(\alpha_X)_{X\in\cU}\mapsto (\Ad\dot u_{\alpha_X})_{X\in\cU}
  \]
  is well-defined and one-to-one by Lemma~\ref{lem:welldefined}.
  Hence we may let $\Phi$ be the corresponding element of $\Aut(\cq(A))$
  given by Lemma~\ref{lem:patch}.
  %
\end{proof}

%

\section{A stronger result}
\label{S.Weak}

In this section, we show it is possible to establish the conclusion of
Theorem~\ref{thm:corona} using a hypothesis that is slightly weaker
than Hypothesis~\ref{hyp:A}.  This will be done by establishing an
analog of each of the lemmas from the previous section.

\begin{hypothesis}
  \label{hyp:weak}
  We assume $A$ has a sequence of positive elements $r_i$, for
  $i\in \bbN$, such that:
  \begin{itemize}
  \item for all $i,j,k$ and all $\epsilon>0$ there exists $a\in A$
    such that $\|a\|=1$ and $\norm{r_i^k\,a\,r_j^k}\geq1-\epsilon$.
  \item the sequence of partial sums $p_n=\sum_{i<n}r_i$ form an
    increasing approximate unit for $A$ with $p_{n+1}p_n=p_n$ for all
    $n$.
  \end{itemize}
\end{hypothesis}

Using the property that $p_{n+1}p_n=p_n$, it is easy to see that
$r_ir_j=0$ whenever $\abs{i-j}\geq2$.  Also, if we define
\[p_I=\sum_{i\in I}r_i
\]
for $I$ an interval of $\NN$, then we similarly have (the definition
of $I(X,i)$ is given before Definition~\ref{D.1.2})
$p_{I(X,i)}p_{I(X,j)}=0$ whenever $\abs{i-j}\geq2$.

\begin{theorem}
  \label{thm:corona-weak}
  Suppose that $\mathcal U\subset\mathcal P(\NN)$ satisfies
  Hypothesis~\ref{hyp:U} and that $A$ satisfies
  Hypothesis~\ref{hyp:weak}.  Then there is an embedding of
  $\ilim_{X\in\cU}\bfG_X$ into the automorphism group of $\cq(A)$.
\end{theorem}

In the proof of Theorem~\ref{thm:corona-weak}, we will not use the
first condition of Hypothesis~\ref{hyp:weak} directly, but rather the
following consequence of it.

\begin{lemma}
  \label{L.epsilon}
  Suppose that the sequence $r_i$ satisfies the first condition in
  Hypothesis~\ref{hyp:weak}.  Then for every $i,j$ and $\epsilon>0$
  there exists $a\in A$ such that $\norm{a}=1$,
  $\norm{r_i\,a\,r_j}\geq1-\epsilon$, and
  $\norm{r_i\,a\,r_j-a}<\epsilon$.
\end{lemma} 

\begin{proof} 
  By the continuous functional calculus, for every contraction $r$ we
  have $r^{k+1}-r^k\rightarrow0$.  It follows that for every
  $\delta>0$ there exists $k$ large enough so that
  $\norm{r_{i}^{k+1}-r_{i}^k}\leq\delta$ and
  $\norm{r_{j}^{k+1}-r_{j}^k}\leq\delta$.  Now, choose $a_0$ such that
  $\norm{a_0}=1$ and $\norm{r_i^{k+1}\,a_0\,r_j^{k+1}}\geq 1-\delta$.
  Then it is easy to see that the element $a=r_i^k\,a_0\,r_j^k$
  satisfies $\norm{r_i\,a\,r_j - a}\leq2\delta$ and
  $\norm{r_i\,a\,r_j}\geq1-\delta$.  Since the $r_i$ are
  norm-decreasing, we of course have $\norm{a}\geq1-\delta$ as well.
  It follows that we can renormalize $a$ and choose $\delta$ small
  enough to obtain the desired inequality.
\end{proof} 

We begin the proof of Theorem~\ref{thm:corona-weak} by again defining
the system of multipliers
\[DD_X(A)=\set{m\in M(A)\mid p_{I(X,i)}\,m\,p_{I(X,j)}=0
  \textrm{ whenever }\abs{i-j}\geq2}\;,
\]
and let $C_X(A)$ denote the image of $DD_X(A)$ in $\cq(A)$.  The next
result again shows that the $C_X(A)$ stratify all of $\cq(A)$.  The
proof is formally identical to that of Lemma~\ref{lem:stratify}.

\begin{lemma}
  For all $m\in M(A)$ there exists a subset $X\subset\NN$ such that
  $m\in DD_X(A)+A$.
\end{lemma}


The next result gives the slight strengthening of Lemma~\ref{lem:key}
necessary for our situation.  Note that the difficulty stems from the
fact that this time, $u_\alpha$ is not necessarily a unitary element of
$A_0$.

\begin{lemma}
  Suppose that $A_0$ is a C*-algebra containing a sequence of positive
  elements $q_0,\ldots,q_{n-1}$ such that
  \begin{itemize}
  \item for all $i,j,k$ and all $\epsilon>0$ there exists $a_0\in A_0$
    such that $\|a_0\|=1$ and
    $\norm{r_i^k\,a_0\,r_j^k}\geq1-\epsilon$, and
  \item $\sum_{i<n}q_i=1$. 
  \end{itemize}
  Letting $I=\set{0,\ldots,n-1}$, for each $\alpha\in\TT^I$ we set
  $u_\alpha=\sum_{i\in I}\alpha(i)q_i$.  Then we have
  \[\Delta_I(\alpha,1)\leq\norm{\Ad u_\alpha-\id}\leq2\Delta_I(\alpha,1)\;.
  \]
\end{lemma}

\begin{proof}
  The proof used in Lemma~\ref{lem:key} again shows that $\norm{\Ad
    u_\alpha-\id}\leq2\Delta_I(\alpha,1)$.

  For the inequality $\Delta_I(\alpha,1)\leq\norm{\Ad u_\alpha-\id}$,
  we first fix $i_0,j_0$ and for any $a\in A_0$ we write:
  \begin{align}
    u_\alpha\,a\,u_\alpha^*-a
    &=\sum\alpha(i)\overline{\alpha(j)}q_i\,a\,q_j-\sum q_iq_j\notag\\
    &=\sum\left(\alpha(i)\overline{\alpha(j)}-1\right)q_i\,a\,q_j\notag
    \label{eq:key}\\
    &=\left(\alpha(i_0)\overline{\alpha(j_0)}-1\right)q_{i_0}\,a\,q_{j_0}
    +\sum_{\substack{i\neq i_0\\j\neq j_0}}
    \left(\alpha(i)\overline{\alpha(j)}-1\right)q_i\,a\,q_j\;.
  \end{align}
  Given $\epsilon$, we apply Lemma~\ref{L.epsilon} to choose $a$ such
  that $\norm{a}=1$, $\norm{q_{i_0}\,a\,q_{j_0}-a}<\epsilon$ and
  $\norm{q_{i_0}\,a\,q_{j_0}}\geq1-\epsilon$.  Note that the latter
  inequality implies that the whole right-hand term of
  Equation~\eqref{eq:key} is very small.  Indeed, it follows that
  $\norm{q_i\,a\,q_j}<\epsilon$ whenever $i\neq i_0,j\neq j_0$, and
  hence that this last term is bounded by $n\epsilon$.

  These computations imply that the expression in
  Equation~\eqref{eq:key} can be made arbitrarily close to
  $(\alpha(i_0)\overline{\alpha(j_0)}-1)a$, and it follows that
  $\norm{\Ad
    u_\alpha-\id}\geq\abs{\alpha(i_0)\overline{\alpha(j_0)}-1}$.
  Since this is true for all $i_0,j_0\in I$, we can conclude that
  $\norm{\Ad u_\alpha-\id}\geq\Delta_I(\alpha,1)$, as desired.
\end{proof}

Now for each $\alpha\in\TT^\NN$, we again define the corresponding
elements $u_\alpha$ of $M(A)$ by
\[u_\alpha=\sum_{i\in\NN}\alpha(i)r_i\;.
\]
These need not be unitaries, but the following result shows that for
plenty of $\alpha$, the image $\dot u_\alpha$ in $\cq(A)$ will in fact
be unitary.

\begin{lemma}
  If $\alpha\in\TT^\NN$ satisfies $\alpha(i+1)-\alpha(i)\rightarrow0$
  then $\dot u_\alpha$ is a unitary in $\cq(A)$.
\end{lemma}

\begin{proof}
  Recall that $r_ir_j=0$ for $\abs{i-j}\geq2$.  Hence we have:
  \begin{align*}
    u_\alpha u_\alpha^*-1
    &=\sum\alpha(i)r_i\sum\overline{\alpha(i)}r_i-\sum r_i\sum r_i\\
    &=\sum\left[2\Re(\alpha(i)\overline{\alpha(i+1)})r_ir_{i+1}+r_i^2\right]
      -\sum\left[2r_ir_{i+1}+r_i^2\right]\\
    &=\sum2\left[\Re(\alpha(i)\overline{\alpha(i+1)})-1\right]r_ir_{i+1}\;.
  \end{align*}
  Since every partial sum lies in $A$, it is enough to show that the
  tails of this last series converge to zero in norm.  By our
  hypothesis, given $\epsilon$, we can find $N$ such that $i>N$
  implies
  \[2\abs{\Re(\alpha(i)\overline{\alpha(i+1)})-1}<\epsilon\;.
  \]
  Since the $r_i$ commute, we can regard them as complex-valued
  functions on the Gelfand space $X$ of $C^*(\set{r_i})$.  Since
  $r_jr_k=0$ whenever $\abs{j-k}\geq2$, each $x\in X$ can only lie in
  the support of at most three of the terms $r_ir_{i+1}$.  Hence we
  can bound the tail
  \[\norm{\sum_{n>N}2\left[\Re(\alpha(i)\overline{\alpha(i+1)})-1\right]
    r_ir_{i+1}}<3\epsilon\;,
  \]
  as desired.
\end{proof}

Thus, if we let $Z\subset\TT^\NN$ be the set of $\alpha$ such that
$\alpha(i+1)-\alpha(i)\rightarrow0$, we have that for all $\alpha\in
Z$ the conjugation $\Ad\dot u_\alpha$ defines an element of
$\Aut(\cq(A))$.  In order to proceed, we must argue that $Z$ is large
enough that there are still $2^{\aleph_1}$ many elements of
$\ilim\bfG_X$ consisting just of elements of $Z$.

\begin{lemma}
  Under the hypotheses of Theorem~\ref{thm:ilim}, there are more than
  continuum many threads through $\ilim\bfG_X$ of the form
  $(\alpha_X)_{X\in\mathcal U}$ where $\alpha_X\in Z$.
\end{lemma}

\begin{proof}
  We must simply inspect the construction given in the proof of
  Theorem~\ref{thm:ilim}, and check that it can be carried out with
  the additional condition:
  \begin{itemize}
  \item $\alpha_s(i)-\alpha_s(i+1)\rightarrow0$.
  \end{itemize}
  For the successor step, this is essentially immediate.  Indeed,
  using the notation of Theorem~\ref{thm:ilim}, recall that the
  witness $\alpha\in\bfF_\xi\minus\bfF_{\xi+1}$ which we produced
  satisfies $\abs{\alpha(n)-\alpha(n+1)}\leq\pi/m$ for $n\in
  I(X_{\xi+1},n(m))$ and $\abs{\alpha(n)-\alpha(n+1)}=0$ elsewhere.

  For the inductive step, recall that given $\alpha_n$ we constructed
  a set $X_\infty$, an element $\alpha_s$, and constants $\gamma_n$
  such that $\alpha_s(i)=\gamma_n\alpha_n(i)$ for all $i\in
  I(X_\infty,n)$.  Assuming additionally that the $\alpha_n$ satisfy
  $\alpha_n(i)-\alpha_n(i+1)\rightarrow0$, we can achieve the same for
  $\alpha_s$ by simply thinning out $X_\infty$ in advance so that
  $\abs{\alpha_n(i)-\alpha_n(i+1)}<1/n$ for all $i\in I(X_\infty,n)$.
\end{proof}

Finally, it is easy to see that the proofs of
Lemma~\ref{lem:welldefined}, Lemma~\ref{lem:patch}, and the conclusion
of the proof of Theorem~\ref{thm:corona} together yield an injection
from the set of threads through $\ilim\bfG_X$ consisting of elements
of $Z$ into $\Aut \cq(A)$.  This concludes the proof of
Theorem~\ref{thm:corona-weak}.

\section{Algebras satisfying our hypotheses}
\label{S.Algebras}

In this section, we give a series of conditions on a C*-algebra $A$
which are sufficient to guarantee that $A$ satisfies either
Hypothesis~\ref{hyp:A} or Hypothesis~\ref{hyp:weak}.  In particular,
we complete the proof of the main theorem (Theorem~\ref{thm:main}) by
showing that each of its hypotheses (1)--(4) is sufficient as well.

In this section, we will always assume that $A$ is $\sigma$-unital.

\begin{proposition}
  \label{P1}
  If $A$ has a $\sigma$-unital, non-unital quotient with a faithful
  irreducible representation, then $A$ satisfies
  Hypothesis~\ref{hyp:weak}.
\end{proposition} 

By Theorems~\ref{thm:ilim} and~\ref{thm:corona-weak}, this completes
the proof of Theorem~\ref{thm:main} in the case that
condition (4) holds.  And clearly, condition (1) is a special case of
condition (4).

\begin{proof}[Proof of Proposotion~\ref{P1}]
  Let $\pi\colon A\to \cB(H)$ be an irreducible representation such
  that $\pi[A]$ is non-unital. Let $r_j$, for $j\in \bbN$, be
  contractions of norm 1 such that $p_n=\sum_{j< n} r_j$ for $n\in
  \bbN$ form an approximate unit for $A$.

  To verify Hypothesis~\ref{hyp:weak}, fix $\epsilon>0$, $i<j$ and
  $k$, and choose $\delta$ small enough that
  $(1-\delta)^2/(1+\delta)>1-\epsilon$.  Fix unit vectors $\xi_i$ and
  $\xi_j$ in $H$ such that $\|\xi_i-\pi(r_i^k) \xi_i\|<\delta$ and
  $\|\xi_j-\pi(r_j^k)\xi_j \|<\delta$.  By Kadison's Transitivity
  Theorem (see for instance \cite[II.6.1.12]{Black:Operator}) we can
  find $a\in A$ of norm $\leq 1+\delta$ such that $\pi(a r_i^k)
  \xi_i=(1-\delta) \xi_j$.  Then
  \[\norm{r_j^k\,a\,r_i^k}\geq\norm{\pi(r_j^k\,a\,r_i^k)\xi_i}
  \geq (1-\delta)^2>1-\epsilon
  \]
  and it follows that $\frac 1{1+\delta}a$ is as required.
\end{proof} 

\begin{proposition}
  \label{P2}
  If $A$ is simple and has an approximate unit consisting of
  projections, then $A$ satisfies Hypothesis~\ref{hyp:A}.
\end{proposition}

\begin{proof}
  Since $A$ is simple it has a faithful irreducible representation.
  Let $r_i$, for $i\in\bbN$, be projections such that their partial
  sums form an approximate unit for $A$. Since $r_i^k=r_i$ for all $i$
  and all $k\geq 1$, applying the proof of Proposition~\ref{P1} to
  projections $r_i$, for $i\in \bbN$, we verify that they satisfy
  Hypothesis~\ref{hyp:A}.
\end{proof} 

\begin{proposition}
  \label{P3}
  If $A$ satisfies Hypothesis~\ref{hyp:weak} and $B$ is a
  $\sigma$-unital C*-algebra then $A\otimes B$ satisfies
  Hypothesis~\ref{hyp:weak} for any product norm on $A\otimes B$.
\end{proposition}

This completes the proof of the main theorem~\ref{thm:main} in the
case that condition (3) holds.  And again, since Proposition~\ref{P1}
implies that $\mathcal K$ satisfies Hypothesis~\ref{hyp:weak},
condition (2) is a special case of condition (3).

\begin{proof}[Proof of Proposition~\ref{P3}]
  Let $r_i\in A$ witness that $A$ satisfies Hypothesis~\ref{hyp:weak}
  and let $p_n=\sum_{i<n}r_i$.  Let $q_n$ be an increasing approximate
  unit for $B$.  We claim that
  \[
  s_i=p_{i+1}\otimes q_{i+1}\;-\;p_i\otimes q_i
  \]
  witness that $A\otimes B$ satisfies Hypothesis~\ref{hyp:weak}.
  Indeed, fix $i<j$, $k$ and $\epsilon>0$.  Then there is $a\in A$
  such that $\norm{a}=1$ and $\norm{r_i^k\,a\,r_j^k}\geq1-\epsilon$.
  We will show that $a\otimes q_j$ satisfies $\norm{s_i^k(a\otimes
    q_j)\,s_j^k}\geq1$.

  To see this, fix a pure state $\phi$ of $B$ such that
  $|\phi(q_i)|=1$.  Since $q_i\leq q_j$, this implies $|\phi(q_j)|=1$.
  Then $\id\otimes \phi$ is a completely positive, contraction mapping
  from $A\otimes B$ into $A$. A straightforward computation shows that
  \[(\id\otimes\phi)(s_i^k(a\otimes q_j)s_j^k)=(p_{i+1}-p_i)^ka(p_{j+1}-p_j)^k
  \]
  and since $\id\otimes \phi$ is a contraction we conclude that $\|
  s_i^k (a\otimes q_j) s_j^k\|\geq 1-\epsilon$.
\end{proof}

We close this section with the following additional special case.
Here, an element $h$ of a C*-algebra $A$ is \emph{strictly positive}
if $ah\neq 0$ for all nonzero $a\in A$.

\begin{proposition}
  Assume $A$ has a subalgebra $B$ that satisfies
  Hypothesis~\ref{hyp:weak} and that $B$ contains an element which is
  strictly positive in $A$. Then $A$ satisfies
  Hypothesis~\ref{hyp:weak}.
\end{proposition} 

\begin{proof}
  Let $p_i$, for $i\in \bbN$, be an approximate unit for $B$ such that
  $r_i=p_i-p_{i-1}$ (with $p_{-1}=0$) witness that $B$ satisfies
  Hypothesis~\ref{hyp:weak}.  Since $B$ contains an element which is
  strictly positive in $A$, $p_i$ is an approximate unit for $A$.
  Also, $r_i$ for $i\in \bbN$ clearly witness that $A$ satisfies
  Hypothesis~\ref{hyp:weak}.
\end{proof}

\section{Trivial automorphisms} 
\label{S.trivial}

In this section, we discuss the claim in the introduction that
Definition~\ref{Def.trivial} is the most comprehensive definition of a
trivial automorphism of a corona of a separable non-unital C*-algebra.
Several arguments in this section require some standard results from
descriptive set theory (see \emph{e.g.},~\cite{Ke:Classical}).

Recall that if $\Phi$ is an automorphism of $\cq(A)$ we write
\[\Gamma_\Phi=\{(a,b): \Phi(a/A)=b/A\}\;,
\]
and that $\Phi$ is said to be trivial if and only if
$\Gamma_\Phi$ is Borel.  

\begin{lemma}
  Assume $A$ is a separable, non-unital, C*-algebra.  Then the trivial
  automorphisms of $\cq(A)$ form a group.
\end{lemma} 

\begin{proof}
  It is clear that $\Phi$ is trivial if and only if $\Phi^{-1}$ is
  trivial, so we only need to check that the composition of two
  trivial automorphisms $\Phi$ and $\Psi$ is trivial.  So suppose that
  $\Gamma_\Phi$ and $\Gamma_\Psi$ are Borel; we need to check that
  $\Gamma=\Gamma_{\Phi\circ \Psi}$ is Borel.  We shall use the fact
  that a subset of a Polish space is Borel if and only if it has a
  $\mathbf\Pi^1_1$ definition and a $\mathbf\Sigma^1_1$ definition
  (\cite[Theorem~14.11]{Ke:Classical}).  Clearly, a $\Sigma^1_1$
  definition is given by
  \[
  \Gamma=\{(a,c):(\exists b\in M(A))\ (a,b)\in \Gamma_\Psi\text{ and }
  (b,c)\in \Gamma_\Phi\}\;.
  \]
  Furthermore, since $\Phi\circ \Psi$ is an automorphism of $\cq(A)$
  we can write
  \[
  \Gamma=\{(a,c):(\forall b\in M(A))\ (a,b)\notin\Gamma_\Psi\text{ or
  } (b,c)\in \Gamma_\Phi\}
  \]
  which gives a $\Pi^1_1$ definition.
\end{proof} 

It is easy to see that every inner automorphism $\Phi$ of $\cq(A)$ is
trivial.  Indeed, if $\pi\colon M(A)\to \cq(A)$ denotes the quotient
map, let $v\in M(A)$ be such that $\Phi$ is conjugation by $\pi(v)$.
Then $\Gamma_\Phi=\{(a, b) : b-vav^*\in A\}$ is Borel.  We now show
that in the case of the Calkin algebra, an automorphism is inner if
and only if it is trivial (cf. \cite[Theorem~2.6]{Fa:All}).

\begin{lemma}
  \label{L.Calkin}
  An automorphism of the Calkin algebra is trivial if and only if it
  is inner.
\end{lemma} 

\begin{proof}
  We need only to prove the direct implication.  Consider $\cB(H)^2$
  with respect to the product of strict topology.  Let $\Phi$ be an
  automorphism of the Calkin algebra such that $\Gamma_\Phi$ is Borel.
  A straightforward computation shows that for a Borel subset $\Gamma$
  of $\cB(H)^2$ the complexity of the assertion that
  $\Gamma=\Gamma_\Phi$ for some automorphism $\Phi$ of the Calkin
  algebra is at most~$\Pi^1_2$, with a code for $\Gamma$ as a
  parameter.  Hence by Shoenfield's absoluteness theorem
  (\cite[Theorem~13.15]{Kana:Book}), in every forcing extension one
  can use the Borel code for~$\Gamma_\Phi$ to define an automorphism
  of $\cq(H)$. This automorphism is an extension of $\Phi$ and we
  denote it by $\tilde\Phi$.  The assertion that $\tilde\Phi$ is inner
  is $\Sigma^1_2$ with a code for $\Gamma$ as a parameter, and again
  by Shoenfield's absoluteness theorem $\Phi$ is inner if and only if
  $\tilde\Phi$ is inner in all forcing extensions.  By
  \cite[Theorem~1]{Fa:All} there exists a forcing extension in which
  all automorphisms of the Calkin algebra are inner. Therefore $\Phi$
  is inner.
\end{proof}

The situation with coronas of abelian C*-algebras is similar.  In
\cite[\S4]{Fa:AQ} the second author considered trivial homeomorphisms
of \v Cech--Stone remainders of locally compact Polish spaces. Such
$F\colon \beta X\smallsetminus X\to \beta X\smallsetminus X$ is
\emph{trivial} if there are compact subsets $K$ and $L$ of $X$ and a
homeomorphism $f\colon X\smallsetminus K \to X\smallsetminus L$ such
that the continuous extension of $f$ to $\beta X$ agrees with $F$ on
$\beta X\smallsetminus X$. In \cite[\S 4]{Fa:AQ} it was proved that
the assertion ``all homeomorphisms of $\beta X\smallsetminus X$ are
trivial'' is relatively consistent with ZFC for all countable locally
compact spaces $X$.  Forcing axioms conjecturally imply all
homeomorphisms of \v Cech--Stone remainders of locally compact Polish
spaces are trivial (\cite[\S 4]{Fa:AQ}).  The absoluteness argument of
Lemma~\ref{L.Calkin} shows that for such $X$ and $A=\cc_0(X)$ an
automorphism of $\cq(A)$ is trivial if and only if the corresponding
homeomorphism of $\beta X\smallsetminus X$ is trivial.

It would be desirable to define trivial automorphisms of $\cq(A)$ as
those that have a representation with certain algebraic properties.
However, even in the case of inner automorphisms of the Calkin algebra
one cannot expect to have a representation that is an automorphism of
$M(A)$, or even a *-homomorphism of $M(A)$ into itself.  This
situation is analogous to the problem of whether ``topologically
trivial'' automorphisms of quotient Boolean algebras $\cP(\bbN)/I$ are
necessarily ``algebraically trivial'' (see \cite{Fa:Rigidity}).

Unlike \cite{Fa:Rigidity} or \cite{Fa:AQ} where the main theme was the
existence of isomorphisms between quotient structures, we have
considered only automorphisms of corona algebras. One reason for this
is that we are unable to answer the following question.

\begin{question}
  Are there separable non-unital C*-algebras $A$ and $B$ whose coronas
  are isomorphic, but there is no trivial isomorphism between them?
\end{question}

\bibliographystyle{amsplain}
\bibliography{scifbib}
\end{document}